\documentclass[11pt,a4paper]{amsart}

\usepackage{latexsym}
\usepackage{amsmath}
\usepackage{amsthm}
\usepackage{amssymb}
\usepackage{amsfonts}
\usepackage{fancyhdr}
\usepackage{comment}

\usepackage{mathrsfs}                           
\usepackage[ocgcolorlinks, linkcolor=blue]{hyperref}

\usepackage{calc}
             {\begin{list}{\arabic{enumi}.}{\usecounter{enumi}%
              \setlength{\labelsep}{0.5em}%
              \settowidth{\labelwidth}{\arabic{enumi}.}%
              \setlength{\leftmargin}{\labelwidth+\labelsep}}}%
             {\end{list}}

\newcommand{\mR}{\mathbb{R}}                    
\newcommand{\mC}{\mathbb{C}}                    
\newcommand{\R}{\mathbb{R}}                    
\newcommand{\C}{\mathbb{C}}                    
\newcommand{\abs}[1]{\lvert #1 \rvert}          
\newcommand{\norm}[1]{\lVert #1 \rVert}         

\newcommand{\mO}{\mathcal{O}}

\newcommand{\re}{\mathrm{Re}}
\newcommand{\im}{\mathrm{Im}}

\newcommand{\dbar}{\overline{\partial}}

\newcommand{\p}{\partial}

\newcommand{\tbl}{\textcolor{blue}}

\newcounter{sidenote}

\theoremstyle{definition}
\newtheorem{thm}{Theorem}[section]
\newtheorem{prop}[thm]{Proposition}
\newtheorem{cor}[thm]{Corollary}
\newtheorem{lemma}[thm]{Lemma}

\newtheorem{conjecture}{Conjecture}[section]
\newtheorem*{definition}{Definition}

\newtheorem*{remark}{Remark}

\numberwithin{equation}{section}

\title[The linearized Calder\'on problem on complex manifolds]{The linearized Calder\'on problem on complex manifolds}

\author[C. Guillarmou]{Colin Guillarmou}
\address{Laboratoire de Math\'ematiques d'Orsay, Universit\'e Paris-Sud, CNRS, Universit\'e Paris-Saclay, 91405 Orsay, France}
\email{cguillar@math.cnrs.fr}

\author[M. Salo]{Mikko Salo}
\address{University of Jyvaskyla, Department of Mathematics and Statistics, PO Box 35, 40014 University of Jyvaskyla, Finland}
\email{mikko.j.salo@jyu.fi}

\author[L. Tzou]{Leo Tzou}
\address{School of Mathematics and Statistics, University of Sydney, Sydney, Australia}
\email{leo@maths.usyd.edu.au}

\begin{document}

\dedicatory{Dedicated to Carlos Kenig on the occasion of his 65th birthday}

\begin{abstract}
In this note we show that on any compact subdomain of a K\"ahler manifold that admits sufficiently many global holomorphic functions, the products of harmonic functions form a complete set. This gives a positive answer to the linearized anisotropic Calder\'on problem on a class of complex manifolds that includes compact subdomains of Stein manifolds and sufficiently small subdomains of K\"ahler manifolds. Some of these manifolds do not admit limiting Carleman weights, and thus cannot by treated by standard methods for the Calder\'on problem in higher dimensions. The argument is based on constructing Morse holomorphic functions with approximately prescribed critical points. This extends results of \cite{GT2} from the case of Riemann surfaces to higher dimensional complex manifolds.
\end{abstract}

\maketitle

\section{Introduction} \label{sec_intro}

If $(M,g)$ is a compact connected oriented Riemannian manifold with $C^{\infty}$ boundary $\partial M$ and if $q \in C^{\infty}(M)$, we consider the Dirichlet problem for the Schr\"odinger equation,
\begin{equation*}
(\Delta_g + q)u = 0 \text{ in } M, \quad u|_{\partial M} = f.
\end{equation*}
Here $\Delta_g$ is the positive Laplace-Beltrami operator. Assuming that $0$ is not a Dirichlet eigenvalue, one defines the Dirichlet-to-Neumann map (DN map) 
\begin{equation*}
\Lambda_{g,q}: C^{\infty}(\partial M) \to C^{\infty}(\partial M), \ f \mapsto \partial_{\nu} u|_{\partial M}.
\end{equation*}
The question studied here is the unique determination of the potential $q$ from the knowledge of the DN map $\Lambda_{g,q}$, provided that $(M,g)$ is also known:

\begin{conjecture}[Calder\'on problem for Schr\"odinger equation] \label{conjecture_main}
Let $(M,g)$ be a compact Riemannian manifold with smooth boundary, and assume that $q_1, q_2 \in C^{\infty}(M)$ so that $0$ is not a Dirichlet eigenvalue of $\Delta_g + q_j$ in $M$. Then $\Lambda_{g,q_1} = \Lambda_{g,q_2}$ implies that $q_1 = q_2$.
\end{conjecture}

This question is closely related to the inverse problem of Calder\'on \cite{C}, which asks to determine a positive function $\sigma \in L^{\infty}(\Omega)$ in a bounded domain $\Omega \subset \mR^n$ from the DN map $\Lambda_{\sigma}: f \mapsto \sigma \partial_{\nu} v|_{\partial \Omega}$ where $v$ solves $\mathrm{div}(\sigma \nabla v) = 0$ in $\Omega$ with $v|_{\partial \Omega} = f$. Here $\sigma$ corresponds to the electrical conductivity of the medium $\Omega$, and $\Lambda_{\sigma}$ encodes voltage and current measurements on $\partial \Omega$. It is well known that if $\sigma$ is sufficiently regular, the substitution $v = \sigma^{-1/2} u$ reduces the Calder\'on problem to solving Conjecture \ref{conjecture_main} when $M = \overline{\Omega}$ and $g$ is the Euclidean metric.

The Calder\'on problem in Euclidean space has been studied intensively and there is a large literature (see \cite{SU, N, Bu} and the survey \cite{Uhlmann2014}). The Calder\'on problem on Riemannian manifolds, also known as the \emph{anisotropic} or \emph{geometric} Calder\'on problem, is substantially more difficult. If $M$ is two-dimensional (i.e.\ a Riemann surface), Conjecture \ref{conjecture_main} was proved in \cite{GT1, GT2} based on the earlier Euclidean case \cite{Bu} (\cite{GT2} considers the more general partial data case, following the Euclidean result \cite{IUY}). For $\dim(M) \geq 3$, much less is known. If all structures are real-analytic there are positive results \cite{LaU, LTU, GS}, but for general $C^{\infty}$ metrics Conjecture \ref{conjecture_main} is only known for a restricted class of manifolds \cite{DKSaU, DKLS}.

One can also consider the linearized version of Conjecture \ref{conjecture_main}. If $(M,g)$ is fixed, this corresponds to asking whether the linearization of the map $q \mapsto \Lambda_{g,q}$ is injective. The linearized Calder\'on problem (linearized at $q=0$) essentially reduces to the following question:

\begin{conjecture}[Linearized Calder\'on problem for Schr\"odinger equation] \label{conjecture_main2}
Let $(M,g)$ be a compact Riemannian manifold with smooth boundary. If $f \in C^{\infty}(M)$ vanishes to infinite order at $\partial M$ and satisfies 
\[
\int_M f u_1 u_2 \,dV_g = 0
\]
for all $u_j \in C^{\infty}(M)$ satisfying $\Delta_g u_j = 0$ in $M$, then $f \equiv 0$.
\end{conjecture}
Here we assume that $f$ vanishes to infinite order at $\partial M$ since we know by boundary determination  \cite{DKSaU} that $\Lambda_{g,q}$ determines $q|_{\partial M}$ to infinite order, and thus the linearization can be reduced to the case of smooth potentials vanishing to any order at $\partial M$.

The methods of \cite{GT2} and \cite{DKLS} also yield a positive answer to Conjecture \ref{conjecture_main2} if $\dim(M) = 2$ or if $\dim(M) \geq 3$, $(M,g)$ is conformally transversally anisotropic, and the geodesic X-ray transform on the transversal manifold is injective (see \cite{DKLS} for more details). However, even the linearized Calder\'on problem remains open in general when $\dim(M) \geq 3$. We refer to \cite{DKLLS} for a recent result related to recovering transversal singularities, and to \cite{DKSjU09, SjU} for results on the linearized partial data problem in Euclidean space.

In this work, our objective is to extend the powerful Riemann surface methods of \cite{GT2} to some class of higher dimensional complex manifolds. It turns out that the linearized problem can be solved on complex manifolds that have sufficiently many global holomorphic functions. Let us first give some related definitions on complex manifolds with boundary.

\begin{definition}
Let $M$ be a compact complex manifold with $C^{\infty}$ boundary, with $\dim_{\mC}(M) = n$. Define 
\[
\mO(M) = \{ f \in C^{\infty}(M) \,;\, f \text{ is holomorphic in } M^{\mathrm{int}} \}.
\]
\begin{enumerate}
\item[(a)]
$M$ is \emph{holomorphically separable} if for any $x, y \in M$ with $x \neq y$, there is some $f \in \mO(M)$ with $f(x) \neq f(y)$.
\item[(b)] 
$M$ has \emph{local charts given by global holomorphic functions} if for any $p \in M$ there exist $f_1, \ldots, f_n \in \mO(M)$ which form a complex coordinate system near $p$.
\end{enumerate}
\end{definition}

The above definitions often appear in connection with Stein manifolds. Let $X$ be complex manifold without boundary, let $\dim_{\mC}(X) = n$, and let $\mO(X)$ be the set of holomorphic functions on $X$. One says that $X$ is a \emph{Stein manifold} if 
\begin{enumerate}
\item[(i)]
for any $x, y \in X$ with $x \neq y$, there is $f \in \mO(X)$ with $f(x) \neq f(y)$;
\item[(ii)] 
for any $p \in X$ there exist $f_1, \ldots, f_n \in \mO(X)$ which form a complex coordinate system near $p$; and 
\item[(iii)] 
$X$ is \emph{holomorphically convex}, meaning that for any compact $K \subset X$ the holomorphically convex hull 
\[
\widehat{K} = \{ p \in X \,;\, |f(p)| \leq \max_{x \in K} |f(x)| \text{ for all } f \in \mO(X) \}
\]
is also compact.
\end{enumerate}

Examples of Stein manifolds include open Riemann surfaces, domains of holomorphy in $\mC^n$, and complex submanifolds of $\mC^N$ that are closed in the relative topology. Conversely, any Stein manifold of complex dimension $n$ admits a proper holomorphic embedding in $\mC^{2n+1}$ (hence nontrivial Stein manifolds are noncompact). See e.g.\ \cite{Forstneric} for these basic facts on Stein manifolds.

The next result gives a positive answer to the linearized Calder\'on problem on complex manifolds that satisfy (a) and (b) above.

\begin{thm} \label{thm_main}
Let $M$ be a compact complex manifold with $C^{\infty}$ boundary. Assume that $M$ is holomorphically separable and has local charts given by global holomorphic functions, and let $g$ be any K\"ahler metric on $M$. If $f \in C^{\infty}(M)$ vanishes to infinite order at $\partial M$ and satisfies 
\[
\int_M f u_1 u_2 \,dV_g = 0
\]
for all $u_j \in C^{\infty}(M)$ satisfying $\Delta_g u_j = 0$ in $M$, then $f \equiv 0$.
\end{thm}

Here are examples of manifolds covered by the above theorem:
\begin{enumerate}
\item[1.] 
If $X$ is a Stein manifold and $g$ is a K\"ahler metric on $X$ (for instance $g$ could be the metric induced by the embedding of $X$ in some $\mC^{N}$), then any compact $C^{\infty}$ subdomain $M$ of $X$ satisfies the conditions in the theorem.
\item[2.] 
More generally, the condition (iii) in the definition of Stein manifolds is not required. For example, any complex submanifold of $\mC^N$ satisfies (i) and (ii) since the functions $f$ and $f_j$ can be constructed from the coordinate functions $z_1, \ldots, z_N$ in $\mC^N$. Hence any compact $C^{\infty}$ subdomain $M$ of some complex submanifold of $\mC^N$, equipped with a K\"ahler metric $g$, satisfies the conditions in the theorem.
\item[3.]
Let $(X,g)$ be a K\"ahler manifold, and let $U$ be a complex coordinate neighborhood in $X$. If $M$ is a compact $C^{\infty}$ subdomain of $U$, then $(M,g)$ satisfies the required conditions: if $f_1, \ldots, f_n$ are complex coordinates in $U$, the conditions (a) and (b) hold just by using the functions $f_1, \ldots, f_n$.
\end{enumerate}

Most of the earlier progress in Conjectures \ref{conjecture_main} and \ref{conjecture_main2} is based on the notion of \emph{limiting Carleman weights} (LCWs), introduced in the Euclidean case in the fundamental work of Kenig, Sj\"ostrand and Uhlmann \cite{KSU}. The notion of LCWs placed the method of complex geometrical optics solutions, which has been used intensively since \cite{SU}, in an abstract general context. LCWs have been particularly useful for partial data problems \cite{KSU, IUY, GT2, KS}, see also the surveys \cite{GT_survey, KS_survey}. The study of LCWs in the geometric case was initiated in \cite{DKSaU} and applied further e.g.\ in \cite{KSaU, KSaU_reconstruction, DKS}. In \cite{DKSaU} it was also shown that the existence of an LCW (at least with nonvanishing gradient) is locally equivalent to a certain conformal symmetry that a generic manifold in dimensions $\geq 3$ will not satisfy \cite{LS, A}. This conformal symmetry has been studied further in \cite{AFGR, AFG} in terms of the structure of the Cotton or Weyl tensors of the manifold.

One of the results in \cite{AFGR} states that $\mC P^2$ with the Fubini-Study metric $g$ does not admit an LCW near any point. Since $(\mC P^2, g)$ is a K\"ahler manifold, compact $C^{\infty}$ subdomains in $\mC P^2$ provide examples of manifolds where Theorem \ref{thm_main} applies but where methods based on LCWs fail.

The proof of Theorem \ref{thm_main} extends the arguments in \cite{Bu} for domains in $\mC$ and \cite{GT2} for Riemann surfaces to the case where $\dim_{\mC}(M) \geq 2$. Here is an outline of the argument:

\begin{enumerate}
\item[1.]
Since $g$ is a K\"ahler metric, we have the factorization 
\[
\Delta_g = 2 \dbar^* \dbar = 2 \partial^* \partial.
\]
Thus holomorphic and antiholomorphic functions are harmonic.
\item[2.] 
We show that there is a dense subset $S$ of $M$ such that for any $p \in S$, there exists a Morse holomorphic function $\Phi$ in $M$ having a critical point at $p$. This is based on the fact that $M$ has local charts given by global holomorphic functions. This allows us to construct a global holomorphic function $\Phi_0$ in $M$ having a prescribed nondegenerate critical point, and further to find Morse functions arbitrarily close to $\Phi_0$ using a transversality argument.
\item[3.] 
The next step is to show that if $\Phi$ is a Morse holomorphic function with critical point at $p$, there is a holomorphic amplitude $a$ such that $a(p) = 1$ and $a$ vanishes at all other critical points of $\Phi$. This follows from the fact that $M$ is holomorphically separable.
\item[4.] 
Finally, for $h > 0$ we define 
\begin{align*}
u_1 &= e^{\Phi/h} a, \\
u_2 &= e^{-\bar{\Phi}/h} \bar{a}.
\end{align*}
Here $u_1$ is holomorphic and $u_2$ is antiholomorphic, hence both are harmonic, and one has 
\[
0 = \int_M f u_1 u_2 \,dV = \int_M f e^{2 i \im(\Phi)/h} |a|^2 \,dV.
\]
Since $\im(\Phi)$ is a Morse function having critical point at $p \in S$, and since $a(p) = 1$ but $a$ vanishes at the other critical points, letting $h \to 0$ (after multiplying by a suitable power of $h$) and applying the stationary phase argument implies that $f(p) = 0$. Since this holds for all $p$ in the dense set $S$, we obtain $f \equiv 0$.
\end{enumerate}

Our argument for treating the linearized Calder\'on problem relies on the fact that we can produce solutions of $\Delta_g u = 0$ out of holomorphic or antiholomorphic functions. Dealing with the full nonlinear Calder\'on problem would require solutions to the Schr\"odinger equation $(-\Delta_g + q) u = 0$, which are typically obtained via Carleman estimates when LCWs are present. However, the manifolds that we are considering do not necessarily admit LCWs, and thus at the moment the methods in this paper are restricted to the linearized problem.

This paper is organized as follows. Section \ref{sec_intro} is the introduction. Section \ref{sec_preliminaries} includes some preliminaries related to complex manifolds, and Section \ref{sec_linearized_inverse} gives the proof of Theorem \ref{thm_main}.

\subsection*{Acknowledgements}

The authors would like to thank Gunther Uhlmann for helpful discussions related to this topic. C.G.\ is partially supported by ERC Consolidator Grant IPFLOW. M.S.\ was supported by the Academy of Finland (Finnish Centre of Excellence in Inverse Problems Research, grant numbers 284715 and 309963) and by the European Research Council under FP7/2007-2013 (ERC StG 307023) and Horizon 2020 (ERC CoG 770924). 

\section{Preliminaries} \label{sec_preliminaries}

In this section we recall standard facts concerning complex and K\"ahler manifolds. We refer to \cite{Huybrechts, Moroianu} for more details.

\subsection{Complex manifolds}

\begin{definition}
An \emph{$n$-dimensional complex manifold} is a $2n$-dimensional smooth (real) manifold with an open cover $U_{\alpha}$ and charts $\varphi_{\alpha}: U_{\alpha} \to \mC^n$ such that $\varphi_{\beta} \circ \varphi_{\alpha}^{-1}$ is holomorphic $\varphi_{\alpha}(U_{\alpha} \cap U_{\beta}) \to \mC^n$.
\end{definition}

\begin{definition}
If $M$ is a differentiable manifold, an \emph{almost complex structure} on $M$ is a $(1,1)$ tensor field $J$ such that the restriction $J_p: T_p M \to T_p M$ satisfies $J_p^2 = -\text{Id}$ for any $p$ in $M$.
\end{definition}

If $M$ is a complex manifold, let $z = (z_1, \ldots, z_n)$ be a holomorphic chart $U_{\alpha} \to \mC^n$, and write $z_j = x_j + i y_j$. There is a canonical almost complex structure $J$ on $M$, defined for holomorphic charts by 
\begin{equation*}
J\left( \frac{\partial}{\partial x_j} \right) = \frac{\partial}{\partial y_j}, \qquad 
J\left( \frac{\partial}{\partial y_j} \right) = -\frac{\partial}{\partial x_j}.
\end{equation*}
Conversely, if $M$ is a differentiable manifold equipped with an almost complex structure $J$ (so it is necessarily even dimensional and orientable), then by the Newlander-Nirenberg theorem $M$ has the structure of a complex manifold if $J$ satisfies an additional integrability condition.

Let $M$ be a complex manifold. From now on we denote by $T_p M$ the complexified tangent space, and by $T_p^* M$ the complexified cotangent space.  We also denote by $J$ the $\mC$-linear extension of the almost complex structure. Locally in holomorphic coordinates, 
\begin{equation*}
T_p M = \mC \left\{ \frac{\partial}{\partial z_j}, \frac{\partial}{\partial \bar{z}_j} \right\}_{j=1}^n, \qquad T_p^* M = \mC \left\{ d z_j, d \bar{z}_j \right\}_{j=1}^n,
\end{equation*}
where 
\begin{equation*}
\frac{\partial}{\partial z_j} =\frac{1}{2} \left( \frac{\partial}{\partial x_j} - i \frac{\partial}{\partial y_j} \right), \qquad \frac{\partial}{\partial \bar{z}_j} = \frac{1}{2} \left( \frac{\partial}{\partial x_j} + i \frac{\partial}{\partial y_j} \right),
\end{equation*}
and the dual basis is 
\begin{equation*}
dz_j = dx_j + i dy_j, \qquad d\bar{z}_j = dx_j - i dy_j.
\end{equation*}

Since $J^2 = -\text{Id}$, the eigenvalues of $J$ acting on $T_p M$ are $\pm i$. We define 
\begin{equation*}
T_p^{1,0}M = \text{Ker}(J-i), \qquad T_p^{0,1} M = \text{Ker}(J+i).
\end{equation*}
These are the holomorphic and antiholomorphic tangent spaces. Locally 
\begin{equation*}
T_p^{1,0} M = \mC \left\{ \frac{\partial}{\partial z_j} \right\}_{j=1}^n, \qquad T_p^{0,1} M = \mC \left\{ \frac{\partial}{\partial \bar{z}_j} \right\}_{j=1}^n.
\end{equation*}
Since $J$ also acts on $T_p^* M$, we define $T_p^{1,0*} M$ and $T_p^{0,1*} M$ in the same way and have that locally 
\begin{equation*}
T_p^{1,0*} M = \mC \left\{ dz_j \right\}_{j=1}^n, \qquad T_p^{0,1*} M = \mC \left\{ d\bar{z}_j \right\}_{j=1}^n.
\end{equation*}

We move on to differential forms. Because of the splitting $T_z^* M = T_z^{1,0*} M \oplus T_z^{0,1*} M$, the set $A^k(M) = C^{\infty}(M ; \Lambda^k M)$ of complex valued $k$-forms on $M$ splits as 
\begin{equation*}
A^k (M) = \bigoplus_{p+q=k} A^{p,q} (M)
\end{equation*}
where $A^{p,q} (M) = C^{\infty}(M ; \Lambda^{p,q} M)$ with 
\begin{equation*}
\Lambda^{p,q}_z M = \left( \bigwedge^p T^{1,0*}_z M \right) \wedge \left( \bigwedge^q T^{0,1*}_z M \right).
\end{equation*}
Locally, any form $u \in A^{p,q}(M)$ can be written as 
\begin{equation*}
u = \sum_{\abs{J}=p,\abs{K}=q} u_{JK} \,dz^J \wedge d\bar{z}^K
\end{equation*}
where $dz^J = dz_1^{j_1} \wedge \ldots \wedge dz_n^{j_p}$ for $j = (j_1,\ldots,j_p)$ etc. Such forms are said to be \emph{of type $(p,q)$}. Clearly 
\begin{equation*}
d: A^{p,q}(M) \to A^{p+1,q}(M) \oplus A^{p,q+1}(M).
\end{equation*}
If $\pi_{p,q}: A^k(M) \to A^{p,q}(M)$ is the natural projection (here $k = p+q$), we can define the $\partial$ and $\dbar$ operators as follows:

\begin{definition}
Let 
\begin{align*}
\partial: A^{p,q}(M) \to A^{p+1,q}(M), & \quad \partial = \pi_{p+1,q} \circ d, \\
\dbar: A^{p,q}(M) \to A^{p,q+1}(M), & \quad \dbar = \pi_{p,q+1} \circ d.
\end{align*}
\end{definition}

Since $d = \partial + \dbar$ and $d^2 = 0$ on forms of type $(p,q)$, we have $(\partial + \dbar)^2 = 0$ on such forms and consequently 
\begin{equation*}
\partial^2 = \dbar^2 = 0, \qquad \partial \dbar + \dbar \partial = 0.
\end{equation*}

\subsection{Hermitian metrics}

For the inverse problem, we want to have a Laplace-Beltrami operator on $M$. For this one needs a Riemannian metric. In the case of complex manifolds, it is natural to assume a compatibility condition.

\begin{definition}
Let $M$ be a complex manifold with almost complex structure $J$. We say that a Riemannian metric $g$ on $M$ is \emph{compatible} with the almost complex structure if $g(Jv,Jw) = g(v,w)$ for all $v$, $w$.
\end{definition}

If $z \in M$, let $h_z(\,\cdot\,,\,\cdot\,) = (\,\cdot\,,\,\cdot\,)_z = (\,\cdot\,,\,\cdot\,)$ be the sesquilinear extension of $g_z$ to the complexified tangent space. Then 
\begin{gather*}
(v,v) > 0 \text{ unless } v = 0, \\
(v,w) = 0 \text{ whenever } v \in T^{1,0} M, w \in T^{0,1} M, \\
\overline{(v,w)} = (\bar{v}, \bar{w}).
\end{gather*}
Conversely, if $(\,\cdot\,,\,\cdot\,)_z$ is a family on symmetric sesquilinear forms on the complex tangent spaces, satisfying the above three conditions and varying smoothly with $z$, then the restriction to the real tangent bundle is a compatible Riemannian metric.

We call $h$ a \emph{Hermitian metric} on $M$, and $(M,h)$ a \emph{Hermitian manifold}. It naturally induces a metric on the complex cotangent spaces. 
One obtains inner products on the exterior powers and also on the space $L^2 A^k(M)$ of complex $k$-forms with $L^2$ coefficients, 
\begin{equation*}
(u, v) = \int_M (u,v)_z \,dV_g(z).
\end{equation*}
The decomposition 
\begin{equation*}
A^k (M) = \bigoplus_{p+q=k} A^{p,q} (M)
\end{equation*}
is orthogonal with respect to this inner product. We extend the Hodge star operator $\ast$ as a complex linear operator on complex forms. Since then 
\begin{equation*}
(u, v) = *u \wedge *\bar{v},
\end{equation*}
the orthogonality implies that $*$ maps $A^{p,q}(M)$ to $A^{n-q,n-p}(M)$.

The $L^2$ inner product induced by the Hermitian metric allows us to define the adjoints of $\partial$ and $\dbar$ as operators 
\[
\partial^*: A^{p+1,q}(M) \to A^{p,q}(M), \qquad \dbar^*: A^{p,q+1}(M) \to A^{p,q}(M).
\]
In terms of the Hodge star operator they may be expressed as 
\[
\partial^* = -\ast \dbar \,\ast, \qquad \dbar^* = -\ast \partial \ast.
\]

\subsection{K\"ahler manifolds}

If $(M,h)$ is a Hermitian manifold and $g$ is the corresponding Riemannian metric, we would like to factor the Laplace-Beltrami operator $\Delta = \Delta_g=d^*d$ acting on functions in terms of the $\partial$ and $\dbar$ operators. The situation is particularly simple on K\"ahler manifolds.

\begin{definition}
A Hermitian manifold $(M,h)$ is called \emph{K\"ahler} if the almost complex structure $J$ satisfies $\nabla J = 0$.
\end{definition}

There are several equivalent characterizations. If $(M,h)$ is a Hermitian manifold, the \emph{fundamental form} is the alternating $2$-form acting on real tangent vectors by 
\begin{equation*}
\omega(v,w) = g(Jv,w).
\end{equation*}
The manifold is K\"ahler iff $d\omega = 0$ (and then $\omega$ is called a \emph{K\"ahler form}). Clearly, the metric $g$ can be recovered from the K\"ahler form and vice versa. One knows that on a K\"ahler manifold, near any point there is a smooth function $f$ such that 
\begin{equation*}
\omega = i \partial \dbar f.
\end{equation*}
Thus, the metric on a K\"ahler manifold locally only depends on one function.

The important fact for our purposes is the following (this is a special case of the K\"ahler identities when acting on $0$-forms).

\begin{lemma}
If $(M,h)$ is K\"ahler, then the Laplace-Beltrami operator on functions satisfies 
\begin{equation*}
\Delta = 2 \partial^* \partial = 2 \dbar^* \dbar.
\end{equation*}
\end{lemma}

Here $\partial^*$ and $\dbar^*$ are the formal adjoints of $\partial$ and $\dbar$ in the $L^2$ inner product as discussed above. Thus 
\begin{equation*}
\partial^*: A^{1,0}(M) \to C^{\infty}(M), \qquad \dbar^*: A^{0,1}(M) \to C^{\infty}(M).
\end{equation*}
The following immediate consequence is the crucial point for solving the linearized Calder\'on problem.

\begin{lemma}
On a K\"ahler manifold, the real and imaginary parts of holomorphic or antiholomorphic functions are harmonic.
\end{lemma}

\begin{remark}
In the solution of the inverse problem on Riemann surfaces, one needs to construct holomorphic functions in $M$ with prescribed zeros or Taylor series in a finite set of points. This was done in \cite{GT2} by using a version of the Riemann-Roch theorem. There are well known extensions of the Riemann-Roch theorem to higher dimensional complex manifolds, such as the Riemann-Roch-Hirzebruch theorem. In this work, instead of using Riemann-Roch type results, we will construct the required holomorphic functions directly from the assumption that the manifold has local charts given by global holomorphic functions.
\end{remark}

\section{Linearized inverse problem} \label{sec_linearized_inverse}

In this section we will prove Theorem \ref{thm_main}. This will be done by constructing complex geometrical optics solutions to the Laplace equation, obtained from holomorphic functions given in the next proposition.

\begin{prop} \label{prop_complex_cgo}
Let $M$ be a compact complex manifold with $C^{\infty}$ boundary. Assume that $M$ has local charts given by global holomorphic functions, and that $M$ is holomorphically separable. Let also $k \geq 2$. There is a dense subset $S$ of $M$ such that for any point $p \in S$ and for any $h > 0$, there is a holomorphic function 
\[
u = e^{\pm \Phi/h} a
\]
where $\Phi \in C^k(M) \cap \mO(M^{\mathrm{int}})$ satisfies $d\Phi(p) = 0$ and $\im(\Phi)$ is Morse in $M$, and $a \in \mO(M)$ satisfies $a(p) = 1$ and $a(p_1) = \ldots a(p_N) = 0$ where $\{ p, p_1, \ldots, p_N \}$ are the critical points of $\im(\Phi)$ in $M$. The functions $\Phi$ and $a$ are independent of $h$.
\end{prop}

\begin{proof}[Proof of Theorem \ref{thm_main}]
Suppose that $f \in C^{\infty}(M)$ vanishes to infinite order at $\partial M$ and 
\[
\int_M f u_1 u_2 \,dV_g = 0
\]
for all $u_j \in C^{\infty}(M)$ with $\Delta_g u_j = 0$ in $M$. Since any harmonic function in $H^1(M)$ can be approximated by $C^{\infty}$ harmonic functions by smoothing out its boundary data, we may assume that the above identity holds for harmonic functions in $C^k(M)$ where $k \geq 2$.

We now use Proposition \ref{prop_complex_cgo}: for any point $p$ in the dense subset of $S$ and for $h > 0$, we choose holomorphic functions 
\begin{align*}
u_1 &= e^{\Phi/h} a, \\
v_2 &= e^{-\Phi/h} a
\end{align*}
so that $\im(\Phi)$ is Morse with critical points $\{ p, p_1, \ldots, p_N \}$, $a(p) = 1$, and $a(p_1) = \ldots = a(p_N) = 0$. Define $u_2 = \bar{v}_2$. Since $g$ is a K\"ahler metric on $M$, $u_1$ and $u_2$ are harmonic functions in $M$. We obtain that 
\[
\int_M f e^{2i\im(\Phi)/h} |a|^2 \,dV_g = 0
\]
for all $h > 0$. Since $p$ is a critical point of $\im(\Phi)$, since $a(p) = 1$ but $a$ vanishes at all the other critical points of $p$, and since $f$ vanishes to infinite order on $\partial M$, the stationary phase argument implies that $f(p) = 0$. Since this is true for all $p$ in a dense subset of $M$, it follows that $f \equiv 0$.
\end{proof}

The next result is the first step in constructing holomorphic functions with prescribed critical points. We will later need to perturb these functions so that their real and imaginary parts become Morse.

\begin{lemma} \label{lemma_phase_function_first}
Let $M$ be a compact complex manifold with $C^{\infty}$ boundary. Assume that $M$ has local charts given by global holomorphic functions. Then for any $p \in M$, there exists $\Phi \in \mO(M)$ such that $p$ is a nondegenerate critical point of both $\re(\Phi)$ and $\im(\Phi)$.
\end{lemma}

Before the proof, we give an elementary lemma related to critical points. Recall that if $u$ is a $C^{\infty}$ real valued function in a real manifold $M$, then we can define the Hessian as the $2$-tensor $D^2u$ where $D$ is the Levi-Civita connection. If 
for a given point $p$ we have $du(p) = 0$, then the Hessian $D^2 u(p)$ does not depend on the metric and if $x$ are local coordinates near $p$, one has 
\[
D^2 u(p) = \sum_{j,k}\frac{\partial^2 u(p)}{\partial x_j \partial x_k} \,dx^j \otimes dx^k.
\]

\begin{lemma} \label{lemma_hessian_elementary}
Let $M$ be a complex manifold and let $f \in \mO(M)$. Write $f = u + iv$ where $u$ and $v$ are real valued. If $p \in M$, one has 
\[
df(p) = 0 \ \ \Longleftrightarrow \ \ du(p) = 0 \ \ \Longleftrightarrow \ \ dv(p) = 0.
\]
At any point $p$ where $df(p) = 0$ there is a well-defined holomorphic Hessian $D^2_{\mathrm{hol}} f(p)$, which is a symmetric bilinear form on the holomorphic tangent space $T_p^{1,0} M$, such that if $z$ are complex local coordinates near $p$ one has 
\[
D^2_{\mathrm{hol}} f(p) =\sum_{j,k} \frac{\partial^2 f(p)}{\partial z_j \partial z_k} \,dz^j \otimes dz^k.
\]
The form $D^2_{\mathrm{hol}} f(p)$ is nondegenerate  on $T_p^{1,0} M$ iff $D^2 u(p)$ is nondegenerate on $T_p M$ iff $D^2 v(p)$ is nondegenerate on $T_p M$.
\end{lemma}
\begin{proof}
The first claim is obvious. The complex Hessian of $f$ at $p$ is simply given by the $T_p^{1,0*} M\otimes T_p^{1,0*} M$ part of the tensor $D^2f(p)$.
 To prove the second claim, assume that $df(p) = 0$ and extend $D^2 u(p)$ and $D^2 v(p)$ as complex bilinear forms on $\mC T_p M$. Let $z$ be complex local coordinates near $p$. We claim that for $a^j, b^j, c^k, d^k \in \mC$, 
\begin{equation} \label{real_holomorphic_hessian}
(D^2 u + i D^2 v)(a^j \partial_{z_j} + b^j \partial_{\bar{z}_j}, c^k \partial_{z_k} + d^k \partial_{\bar{z}_k}) = \frac{\partial^2 f}{\partial z_j \partial z_k} a^j c^k
\end{equation}
where everything is evaluated at $p$ (we suppress $p$ from the notation here and below).
The formula \eqref{real_holomorphic_hessian} implies that $D^2_{\mathrm{hol}} f = D^2 u + i D^2 v$ is indeed invariantly defined. To prove the claims concerning non-degeneracy, note that $D^2_{\mathrm{hol}} f (\sum_j a^j \partial_{z_j},w) = 0$ for all $w\in T^{1,0}_p M$ implies that for all $w'\in T_p M$
\[
D^2 u(\sum_j a^j \partial_{z_j} + \bar{a}^j \partial_{\bar{z}_j},w') = D^2 v(\sum_ja^j \partial_{z_j} + \bar{a}^j \partial_{\bar{z}_j},w') = 0.
\]
Thus the non-degeneracy of $D^2u$ or $D^2v$ implies non-degeneracy of $D^2_{{\rm hol}}f$.
Conversely, if $D^2 u(\sum_ja^j \partial_{z_j} + \bar{a}^j \partial_{\bar{z}_j},w') = 0$ for all $w'\in T_p M$, then taking $w'=c^k\partial_{z_k}+\bar{c}^k\partial_{\bar{z}_k}$ and using \eqref{real_holomorphic_hessian}  implies that $\re(\sum_{j}\frac{\partial^2 f}{\partial z_j \partial z_k} a^j c^k) = 0$ for all $c^k \in \mC$. Choosing real and imaginary values for $c^k$ implies that $D^2_{\mathrm{hol}} f (\sum_ja^j \partial_{z_j},w) = 0$ for all $w\in T_p^{1,0}M$. A similar statement holds for $D^2 v$.

The proof of \eqref{real_holomorphic_hessian} is straightforward. If $s, t \in \{ 1, -1 \}$, one has 
\begin{align*}
D^2 u(\p_{x_j} + is \p_{y_j}, \p_{x_k} + it \p_{y_k}) &= \p_{x_j}\p_{x_k}u - st \p_{y_j}\p_{y_k}u+ 
i( t \p_{x_j}\p_{y_k}u + s \p_{y_j}\p_{x_k}u).
\end{align*}
Thus 
\[
D^2 u = \sum_{j,k}\frac{\p^2u}{\p_{z_j}\p_{z_k}} \,dz^j \otimes dz^k +\frac{\p^2u}{\p_{z_j}\p_{\bar{z}_k}}  \,dz^j \otimes d\bar{z}^k + \frac{\p^2u}{\p_{\bar{z}_j}\p_{z_k}} \,d\bar{z}^j \otimes dz^k + \frac{\p^2u}{\p_{\bar{z}_j}\p_{\bar{z}_k}} \,d\bar{z}^j \otimes d\bar{z}^k.
\]
We now invoke the fact that $f$ is holomorphic, which implies that 
\begin{align*}
\p_{z_j} u &= \frac{1}{2} \p_{z_j} ( f + \bar{f} ) = \frac{1}{2} \p_{z_j} f, \\
\p_{\bar{z}_j} u &= \frac{1}{2} \p_{\bar{z}_j} ( f + \bar{f} ) = \frac{1}{2} \p_{\bar{z}_j} \bar{f},
\end{align*}
and consequently 
\[
D^2 u = \frac{1}{2} (\sum_{j,k}\p_{z_j}\p_{z_k}f  \,dz^j \otimes dz^k + \overline{\p_{z_j}\p_{z_k} f \,dz^j \otimes dz^k}).
\]
Similarly 
\[
D^2 v = \frac{1}{2i} (\sum_{j,k}\p_{z_j}\p_{z_k} f \,dz^j \otimes dz^k - \overline{\p_{z_j}\p_{z_k} f \,dz^j \otimes dz^k}).
\]
This proves \eqref{real_holomorphic_hessian}.
\end{proof}

\begin{proof}[Proof of Lemma \ref{lemma_phase_function_first}]
Since $M$ has local charts given by global holomorphic functions, we can find functions $f_1, \ldots, f_n \in \mO(M)$ that form a complex coordinate system near $p$. By subtracting constants, we can assume that $f_1(p) = \ldots = f_n(p) = 0$. The required function $\Phi$ will then be given by 
\[
\Phi(z) = \sum_{l=1}^n f_l(z)^2.
\]
Clearly $\Phi \in \mO(M)$ and $\Phi(p) = 0$. 
Note that the differential of $\Phi$ is given by 
\[
d\Phi = 2 \sum_{l=1}^n f_l \,df_l.
\]
Since $f_l(p) = 0$ one has $d\Phi(p) = 0$, so $p$ is a critical point of both $\re(\Phi)$ and $\im(\Phi)$.

The holomorphic Hessian of $\Phi$ in the $(f_1, \ldots, f_n)$ coordinates is given by 
\begin{align*}
D^2_{\mathrm{hol}} \Phi(p) &= 2 \sum_{l=1}^n  df_l \otimes df_l 
\end{align*}
Thus $D^2_{\mathrm{hol}} \Phi(p)(a^j \partial_{f_j}) = 2 a^l df_l$, showing that $D^2_{\mathrm{hol}} \Phi(p)$ is nondegenerate on $T_p^{1,0} M$. Lemma \ref{lemma_hessian_elementary} implies that $p$ is a nondegenerate critical point both for $\re(\Phi)$ and $\im(\Phi)$.
\end{proof}

The next step is to show that we may approximate the functions in Lemma \ref{lemma_phase_function_first} by holomorphic functions whose real and imaginary parts are Morse. For Riemann surfaces such a result was given in \cite{GT2}, based on a transversality argument from \cite{Uhlenbeck}. We will follow the same approach.

\begin{lemma} \label{lemma_phase_function_second}
Let $M$ be a compact complex manifold with $C^{\infty}$ boundary. Assume that $M$ has local charts given by global holomorphic functions. Let also $k \geq 2$. There is a dense subset $S$ of $M$ such that for any $p \in S$, there is $\Phi \in C^k(M) \cap \mO(M^{\mathrm{int}})$ having a critical point at $p$ so that both $\re(\Phi)$ and $\im(\Phi)$ are Morse functions in $M$.
\end{lemma}
\begin{proof}
Define 
\[
H = \{ (\re(f), \im(f)) \,;\, f \in C^k(M) \cap \mO(M^{\mathrm{int}}) \}.
\]
Fix some $C^{\infty}$ Riemannian metric $g$ on $M$ and define 
\[
\norm{u}_{C^k(M)} = \sum_{j=0}^k \norm{\nabla^j u}_{L^{\infty}(M)}
\]
where $\nabla$ is the Levi-Civita connection and $L^{\infty}(M) = L^{\infty}(M, dV_g)$. Equip $H$ with the norm $\norm{(u,v)} = \norm{u}_{C^k(M)} + \norm{v}_{C^k(M)}$, which makes $H$ a Banach space.

We claim that the set 
\begin{equation} \label{u_v_morse}
\{ (u, v) \in H \,;\, \text{$u$ and $v$ are Morse in $M$} \}
\end{equation}
is dense in $H$. If this holds, then for any point $p_0 \in M^{\mathrm{int}}$ we may use Lemma \ref{lemma_phase_function_first} to find $\Phi = \varphi + i \psi \in \mO(M)$ having a nondegenerate critical point at $p_0$. The density of \eqref{u_v_morse} implies that there are $\Phi^{(l)} = \varphi^{(l)} + i \psi^{(l)}$ in $C^k(M) \cap \mO(M^{\mathrm{int}})$ such that $\varphi^{(l)}$ and $\psi^{(l)}$ are Morse and $\varphi^{(l)} \to \varphi$, $\psi^{(l)} \to \psi$ in $C^k(M)$. It follows that $\mathrm{Hess}_g(\varphi^{(l)}) \to \mathrm{Hess}_g(\varphi)$ on $M$. In particular, if $x$ are Riemannian normal coordinates at $p_0$, both matrices $(\partial_{x_j}\p_{x_k} \varphi^{(l)})$ and $(\partial_{x_j}\p_{x_k} \varphi)$ are nondegenerate at $p_0$ for $l$ large, and by the inverse function theorem the vector fields $(\partial_{x_j} \varphi^{(l)})$ and $(\partial_{x_j} \varphi)$ are invertible maps in some neighborhood of $p$ which is independent of $l$. It follows that there exist $p^{(l)} \in M^{\mathrm{int}}$ with $d\varphi^{(l)}(p^{(l)}) = 0$ and $p^{(l)} \to p_0$. This proves that the set 
\begin{multline*}
S = \{ p \in M \,;\, \text{$df(p) = 0$ for some $f \in C^k(M) \cap \mO(M^{\mathrm{int}})$} \\
 \text{with $\re(f), \im(f)$ Morse} \}
\end{multline*}
is dense in $M$. The result then follows.

To prove that \eqref{u_v_morse} is dense in $H$, we repeat the argument in \cite{GT2} and consider the map 
\[
m: H \times M \to T^* M, \ \ m((u,v), p) = (p, du(p)).
\]
We also write $m_{u,v}: M \to T^* M, \ m_{u,v}(p) = (p, du(p))$. We first observe that $p$ is a critical point of $u$ iff $m_{u,v}(p) \in T_0^* M$, where $T_0^* M$ is the zero section of $T^* M$. If $p$ is a critical point, $X \in T_p M$ and if $\gamma(t)$ is a smooth curve in $M$ with $\dot{\gamma}(0) = X$, we compute 
\[
(D_p m_{u,v}) X = \frac{d}{dt} (\gamma(t), du(\gamma(t))) \Big|_{t=0} = (X, D^2 u(X)).
\]
Thus we have proved that for $(u,v) \in H$, 
\begin{align}
 &\text{$u$ is Morse} \label{u_morse_transversality}\\
 &\Longleftrightarrow \text{Ran}(D_p m_{u,v}) + T_{(p,du(p))}(T_0^* M) = T_p M \oplus T_p^* M \text{ when $m_{u,v} \in T_0^* M$}. \notag
\end{align}
This means that $u$ is Morse iff $m_{u,v}$ is transverse to the zero section $T_0^* M$ in $T^* M$ (see \cite{Uhlenbeck} for this terminology).

We next show that the set 
\begin{equation} \label{u_v_residual}
\{ (u,v) \in H \,;\, \text{$m_{u,v}$ is transverse to $T_0^* M$} \}
\end{equation}
is residual (i.e.\ a countable intersection of open dense sets) in $H$. In fact, this follows from \cite[Transversality theorem 2]{Uhlenbeck}, see also \cite{GT2}, provided that we can show that $m$ is transverse to $T_0^* M$. Let $(u,v) \in H$, $p \in M$ be such that $p$ is a critical point of $u$ (i.e.\ $m(u,v,p) \in T_0^* M$), let $(\hat{u}, \hat{v}) \in H$ and $X \in T_p M$, and let $\gamma$ be a smooth curve in $M$ with $\dot{\gamma}(0) = X$. We compute 
\[
(D_{u,v,p} m)(\hat{u}, \hat{v}, X) = \frac{d}{dt} (\gamma(t), d(u + t \hat{u})(\gamma(t))) \Big|_{t=0} = (X, D^2 u(X) + d \hat{u}(p)).
\]
Thus $m$ is transverse to $T_0^* M$ provided that $\{ d\hat{u}(p) \,;\, (\hat{u}, \hat{v}) \in H \}$ spans $T_p^* M$ for any $p \in M$. This last fact follows since $M$ has local charts given by global holomorphic functions. To see this, let $p \in M$ and choose $f_1, \ldots, f_n \in \mO(M)$ which form a complex chart near $p$. Write $f_j = u_j + i v_j$ and assume that for some $a^j, b^j \in \mR$ one has 
\[
\sum_{j=1}^n a^j du_j(p) + b^j dv_j(p) = 0.
\]
Since $2u_j = f_j + \bar{f}_j$ and $2i v_j = f_j - \bar{f}_j$, we obtain 
\[
0 =\sum_{j=1}^n c^j df_j + \bar{c}^j d\bar{f}_j
\]
where $c_j = a_j - i b_j$. Since $\{Êdf_j, d\bar{f}_j \}$ span the complex cotangent space at $p$, we get $c^j = 0$ for $j=1,\dots,n$. Thus $a^j = b^j = 0$, and $\{ du_j(p), dv_j(p) \}_{j=1,\dots,n}$ span the real cotangent space at $p$. The fact that $\re(-if) = \im(f)$ implies that real parts of functions in $\mO(M)$ span $T_p^* M$ as well.

We have now proved that the set \eqref{u_v_residual} is residual. The observation \eqref{u_morse_transversality} and Lemma \ref{lemma_hessian_elementary} show that the set \eqref{u_v_morse} is dense in $H$ as required.
\end{proof}

Finally, we invoke the assumption that $M$ is holomorphically separable to construct the amplitude in Theorem \ref{prop_complex_cgo}.

\begin{proof}[Proof of Proposition \ref{prop_complex_cgo}]
Let $S$ be as in Lemma \ref{lemma_phase_function_second}, let $p \in S$, and let $\Phi \in C^k(M) \cap \mO(M^{\mathrm{int}})$ be the function given in Lemma \ref{lemma_phase_function_second} so that $d\Phi(p) = 0$ and $\im(\Phi)$ is Morse. Let $\{ p, p_1, \ldots, p_N \}$ be the critical points of $\im(\Phi)$ in $M$. Since $M$ is holomorphically separable, for each $j$ we may find $a_j \in \mO(M)$ such that $a_j(p) = 1$ and $a_j(p_j) = 0$. It is enough to choose $a = a_1 \cdots a_N$.
\end{proof}

\providecommand{\bysame}{\leavevmode\hbox to3em{\hrulefill}\thinspace}
\providecommand{\href}[2]{#2}

\bibliographystyle{alpha}

\end{document}